\documentclass[12pt]{amsart}
\usepackage[linkcolor=blue, citecolor = blue]{hyperref}
\hypersetup{colorlinks=true}

\usepackage{amsthm,amsfonts,amsmath,amssymb,amscd} 

\usepackage{verbatim}
\usepackage{float}
\usepackage{wrapfig}
\usepackage{mathtools}
\usepackage[color=orange!20, linecolor=orange]{todonotes}

\usepackage{caption, subcaption}
\usepackage{enumitem}

\usepackage{ytableau}

\usepackage{wasysym}

\newtheorem{theorem}{Theorem}[section]

\newtheorem{lemma}[theorem]{Lemma}

\theoremstyle{definition}

\newtheorem{problem}[theorem]{Problem}

\theoremstyle{remark}
\newtheorem{remark}{Remark}

\title[Bounds on the number of higher-dimensional partitions]{Bounds on the number of higher-dimensional partitions} 

\author[Damir Yeliussizov]{Damir Yeliussizov}

\address{KBTU, Almaty, Kazakhstan}
\email{\href{mailto:yeldamir@gmail.com}{yeldamir@gmail.com}}

\begin{document}

\begin{abstract}
We establish some bounds on the number of higher-dimensional partitions by volume. 
In particular, we give bounds via vector partitions and MacMahon's numbers. 
\end{abstract}

\maketitle


\section{Introduction}

{Higher-dimensional partitions} were introduced by P. A. MacMahon \cite{macmahon} as a natural generalization of  
(one-dimensional) 
integer partitions and (two-dimensional) 
plane partitions. 
While these  objects 
arise in different 
fields such as 
algebra, geometry, combinatorics,  and statistical physics 
(
see e.g. 
\cite{mr, bgp, bbs, gov, ms, nekrasov} on some aspects), 
not much is known about 
higher-dimensional partitions 
in dimensions three or more (cf. \cite[\textsection 7.20]{sta}), including 
their asymptotics. 

\vspace{0.5em}

In this note, we prove some asymptotic lower and upper bounds on the number of $d$-dimensional partitions of volume $n$.


\subsection{Definition} 
A {\it $d$-dimensional partition} is a hypermatrix 
$\pi = (\pi_{i_1,\ldots, i_d})_{i_1,\ldots, i_d \ge 1}$ of nonnegative integers with only finitely many nonzero entries such that 
$
\pi_{i_1,\ldots, i_d} \ge \pi_{j_1,\ldots, j_d} \text{ for all } i_1 \le j_1, \ldots, i_d \le j_d.
$
The {\it volume} (or {\it size}) of $\pi$ is $|\pi| = \sum_{i_1,\ldots, i_d} \pi_{i_1,\ldots, i_d}.$ 

\vspace{0.5em}

For $d = 1$, this definition corresponds to the usual integer partitions. For $d = 2$, it defines 
{plane partitions}. We refer e.g. to \cite{sta2, andrews, sta, krat} on the theory of these well-studied objects.

\vspace{0.5em}

Let $\mathsf{p}_d(n)$ be the number of $d$-dimensional partitions of volume $n$. 

\subsection{MacMahon's conjecture and asymptotics} 
{\it MacMahon's numbers} $\mathsf{m}_d(n)$ are defined via the generating series
$$
\sum_{n \ge 0} \mathsf{m}_d(n)\, x^n = \prod_{n \ge 1} (1 - x^{n})^{-\binom{n + d - 2}{d - 1}}.
$$
MacMahon \cite{macmahon} conjectured that $\mathsf{p}_d(n) = \mathsf{m}_d(n)$ which is true for $d = 1,2$, but turns out to be false for $d \ge 3$ as shown 
in \cite{atkin}; 
in fact, 
$
\mathsf{m}_d(6) - \mathsf{p}_d(6) = \binom{d}{3} + \binom{d}{4}
$
(see \cite{andrews}).


An interpretation of MacMahon's numbers via $d$-dimensional partitions was given in \cite{ay}, where it is shown that (instead of the volume) $\mathsf{m}_d(n)$ computes 
$d$-dimensional partitions by another statistic called the {\it corner-hook volume}, which we recall in  \textsection\ref{secmac}.

\vspace{0.5em}

Throughout the paper, we consider the asymptotics as $n \to \infty$ assuming $d$ is fixed.

It is known that (see \cite{bgp})
$$
\lim_{n \to \infty} \frac{\log \mathsf{m}_d(n)}{n^{d/(d  + 1)}} = \gamma_d = {{d^{-d/(d+1)} (d+1)\, \zeta(d+1)^{1/(d + 1)}}},
$$
where $\zeta(s) = \sum_{n \ge 1} n^{-s}$ is the Riemann zeta function. (Using the explicit generating function, a more detailed asymptotics of $\mathsf{m}_d(n)$ can be found.) 

Concerning the asymptotics of $\mathsf{p}_d(n)$, 
it is 
known from 
\cite{bhatia} that $\log \mathsf{p}_{d}(n) = \Theta(n^{d/(d+1)})$. 
(Note that no explicit generating function is known for $\mathsf{p}_d(n)$ when $d \ge 3$.) 

Even though MacMahon's conjecture is false, based on some numerical simulations 
for (3-dimensional) solid partitions, 
it was conjectured in \cite{mr} 
that it might hold asymptotically, i.e. $\log \mathsf{p}_3(n) \sim \log \mathsf{m}_3(n)$. 
Based on experiments extended to higher dimensions, it was conjectured more generally in \cite{bgp} that
$\log \mathsf{p}_d(n) \sim \log \mathsf{m}_d(n) \sim \gamma_d\, n^{d/(d + 1)}.$
However, later experiments for solid partitions reported in \cite{dg} suggested that the conjecture of \cite{mr} appears to be false and $\mathsf{p}_3(n)$ grows faster than $\mathsf{m}_3(n)$.


\subsection{Asymptotic bounds}



We prove the following asymptotic bounds.

\begin{theorem}\label{thm:one}
For all sufficiently large $n$, 
we have the following bounds: 
$$
\alpha_d \, {n^{d/(d+1)}} \le {\log \mathsf{p}_d(n)}  \le \beta_d \, {n^{d/(d+1)}} + d \log n, 
$$
where the constants are
$$
\alpha_d = \frac{(d+1)}{(d+1)!^{1/(d+1)}}  
\cdot \delta_d ~\text{ for some } \delta_d > \log 2, 
\qquad \beta_d = (d+1) \zeta(d+1)^{1/(d + 1)}. 
$$
\end{theorem}

This theorem 
improves an asymptotic result of \cite{bhatia} where it was shown that $\log \mathsf{p}_{d}(n) = \Theta(n^{d/(d+1)})$ without explicit bounds (which can be computed to give much wider range for the constants). 

After an earlier version of 
this paper was written, we learned about the very recent work \cite{dai},
where a slightly smaller lower bound constant $\alpha'_d = \frac{(d+1)}{(d+1)!^{1/(d+1)}} \log 2 - o(1)$ and 
a significantly larger upper bound  $\gamma_1 (d+1)^{\log(d+1)}$ 
(which optimizes the bound from \cite{bhatia}) were obtained. 
For the upper bound 
our approach
is different from \cite{bhatia, dai}; for the lower bound 
our estimate 
is similar to \cite{dai}, but the technique is different giving the constant $\alpha_d$ (cf. Theorem~\ref{coral}).

Comparing the lower bound with MacMahon's numbers, we have $\alpha_d > \gamma_d$ for $d \ge 7$  which 
shows 
that {$\mathsf{p}_d(n)$ is asymptotically larger than $\mathsf{m}_d(n)$ in dimensions $d \ge 7$}, i.e. 
$$
\liminf_{n \to \infty} 
\frac{\log \mathsf{p}_{d}(n)}{n^{d/(d+1) }} > \gamma_d 
$$
and hence the conjecture of \cite{bgp} is false.
We note that this was shown in \cite{dai} using the constant $\alpha'_d$. 
In dimensions $3 \le d \le 6$, an asymptotic comparison of $\mathsf{p}_d(n)$ and $\mathsf{m}_d(n)$ remains unknown.

\begin{remark}
Interestingly, all available numerical values of $\mathsf{p}_d(n)$ (see \cite{atkin, knuth, boltzmann, ekhad} and \cite[A000293, A000294]{oeis}) show that $\mathsf{m}_d(n) \ge \mathsf{p}_d(n)$ and that the sequences grow relatively close. The above limit inequality 
tells that this is delusive for asymptotics. 
\end{remark}


\begin{remark}
It is also open to prove the existence of 
$\lim_{n \to \infty } {n^{-d/(d+1)}} {\log \mathsf{p}_d(n)}$
(which is assumed in some literature referring to \cite{bhatia}). 
\end{remark}

\subsection{An upper bound via vector partitions} 
 
To prove the above results (specifically an upper bound) we give a comparison with numbers whose explicit generating functions are known. 


{\it Vector partition} (or {\it multipartition}) {\it numbers} $\mathsf{p}(n_1,\ldots, n_d)$ are defined via the multivariate generating series 
$$
\sum_{n_1,\ldots, n_d } \mathsf{p}(n_1,\ldots, n_d)\, x_1^{n_1} \cdots x_d^{n_d}  = \prod_{i_1,\ldots, i_d \ge 1} \left(1 - x_1^{i_1} \cdots x_d^{i_d} \right)^{-1}.
$$
(Note that for $d = 1$, we have $\mathsf{p}(n) = \mathsf{p}_1(n)$ is the usual number of partitions.)

\begin{theorem}\label{thmc}
The following inequality holds:
$$
\mathsf{p}_d(n) \le n^{d}\, {\mathsf{p}}(\underbrace{n, \ldots, n}_{d \text{ times}}). 
$$
\end{theorem}

To obtain this
upper bound 
we give a new combinatorial interpretation of vector partition numbers via $d$-dimensional partitions (see Lemma~\ref{wp}).
We also obtain an upper bound via MacMahon's numbers: $$\mathsf{p}_d(n) < dn \cdot \mathsf{m}_d(dn).$$

Our proofs 
use some tools developed in \cite{ay}.

\section{Preliminaries}

We use the following basic notation: $\mathbb{N}$ is the set of nonnegative integers; $\mathbb{Z}_+$ is the set of positive integers; $\{\mathbf{e}_1, \ldots, \mathbf{e}_d\}$ is the standard basis of $\mathbb{Z}^d$; and $[n] := \{1, \ldots, n \}$.

\vspace{0.5em}

A {\it $d$-dimensional $\mathbb{N}$-hypermatrix} is an array $\left(a_{i_1, \ldots, i_d}\right)_{i_1, \ldots, i_d \ge 1}$ of nonnegative integers 
with finite support (i.e. with only finitely many nonzero elements). 
A {\it $d$-dimensional partition} is a $d$-dimensional 
$\mathbb{N}$-hypermatrix 
$\left(\pi_{i_1, \ldots, i_d} \right)_{i_1,\ldots, i_d \ge 1}$ 
such that 
$$
\pi_{i_1, \ldots, i_d} \ge \pi_{j_1, \ldots, j_d} ~\text{ for }~ i_1 \le j_1, \ldots, i_d \le j_d.
$$
Let 
$\mathcal{P}^{(d)}$ be the set of $d$-dimensional partitions. 
For $\pi = (\pi_{i_1, \ldots, i_d}) \in \mathcal{P}^{(d)}$, the {\it volume} (or size) of $\pi$ denoted by $|\pi|$ is defined as 
$$
|\pi| = \sum_{i_1, \ldots, i_d} \pi_{i_1, \ldots, i_d}.
$$

Any partition $\pi$ is uniquely determined by its {\it diagram} $D(\pi)$ which is the set 
$$D(\pi) := \{(i_1, \ldots, i_d, i) \in \mathbb{Z}^{d+1}_+: 1 \le i \le \pi_{i_1, \ldots, i_d} \} 
$$
of cardinality $|D(\pi)| = |\pi|$. 

Let us note that 
for a set $\rho \subset \mathbb{Z}^{d+1}_+$, the following conditions are equivalent:
\begin{itemize}
\item[(1)] The set $\rho$ is the diagram of some $d$-dimensional partition.
\item[(2)] The set $\rho$ is finite and has the property that if $\mathbf{i} \in \mathbb{Z}^{d+1}_+$  and $\ell \in [d+1]$ satisfy $\mathbf{i} + \mathbf{e}_{\ell} \in \rho$, then $\mathbf{i} \in \rho$. 
\end{itemize}
Sometimes we may identify a partition with its diagram. 

\vspace{0.5em}

We denote by $\mathsf{p}_d(n)$ the number of $d$-dimensional partitions of size $n$. Let also $\widetilde{\mathsf{p}}_d(n)$ be the number of $d$-dimensional partitions of size at most $n$, so that for $ n > 1$ we have
$$
\mathsf{p}_d(n) < \widetilde{\mathsf{p}}_d(n) = \sum_{i \le n}^{} \mathsf{p}_d(i) \le n\cdot \mathsf{p}_d(n),
$$
and hence $\log \widetilde{\mathsf{p}}_d(n) = \log \mathsf{p}_d(n) + O(\log n)$ 
has the same main term of asymptotics as $n \to \infty$; in our case we have more precisely
$$\frac{\log \widetilde{\mathsf{p}}_{d}(n)}{n^{d/(d+1)}} = \frac{\log \mathsf{p}_{d}(n)}{n^{d/(d+1)}} + o(1).$$ 

\section{A map from hypermatrices to partitions}\label{secphi}

In this section we describe a map from hypermatrices to partitions and some of its properties studied in \cite{ay} which will be useful for proving our results.

\vspace{0.5em}

A lattice path in $\mathbb{Z}^d$ is called {\it directed} if it uses only steps of the form $\mathbf{i} \to \mathbf{i} + \mathbf{e}_{\ell}$ for $\mathbf{i} \in \mathbb{Z}^d$ and $\ell \in [d]$. 
For a $d$-dimensional $\mathbb{N}$-hypermatrix $A = (a_{i_1, \ldots, i_d})_{i_1, \ldots, i_d \ge 1}$ we define the following map: 
$$\Phi : A \longmapsto \pi = (\pi_{i_1,\ldots, i_d}), \qquad \pi_{i_1,\ldots, i_d} = \max_{\Pi : (i_1,\ldots, i_d) \to (\infty^d)} \sum_{(j_1,\ldots, j_d)\in \Pi } a_{j_1,\ldots, j_d},$$
where the sum is over directed lattice paths $\Pi$ which start at $(i_1,\ldots, i_d)$. Equivalently, we also have 
$$
\pi_{\mathbf{i}} = a_{\mathbf{i}} + \max_{\ell \in [d]} \pi_{\mathbf{i} + \mathbf{e}_{\ell}}, \quad \mathbf{i} \in \mathbb{Z}_+^{d}.
$$
Note that 
$\pi = \Phi(A) \in \mathcal{P}^{(d)}$ is a $d$-dimensional partition.

\vspace{0.5em}


For a partition $\pi \in \mathcal{P}^{(d)}$, define the set of {\it corners} of $\pi$ as follows:
$$
\mathrm{Cor}(\pi) := \left\{\mathbf{i} \in \mathbb{Z}^{d+1}_+ : \mathbf{i} \in D(\pi),\, \mathbf{i} + \mathbf{e}_{\ell} \not\in D(\pi) \text{ for all } \ell \in [d] \right\}.
$$
(Here $\{e_{\ell}\}$ is the standard basis in $\mathbb{Z}^{d+1}$.) For example, if $\pi \in \mathcal{P}^{(2)}$ is a plane partition then $$\mathrm{Cor}(\pi) = \{(i,j,k) : (i,j, k) \in D(\pi), (i+1, j, k), (i, j+1, k) \not\in D(\pi) \}.$$ 


\vspace{0.5em}

The map $\Phi$ has the following important properties.

\begin{theorem}[\cite{ay}] The following properties hold:

(i) The map $\Phi$  is a bijection between $d$-dimensional $\mathbb{N}$-hypermatrices and $d$-dimensional partitions. In particular, it is injective on any set of $\mathbb{N}$-hypermatrices. 

(ii) The inverse map $\Phi^{-1}$ can be described via projection of corners as follows. For $\pi \in \mathcal{P}^{(d)}$ and $A = (a_{i_1,\ldots, i_d}) = \Phi^{-1}(\pi)$, we have 
$$
a_{i_1,\ldots, i_d} = \left| \left\{i : (i_1,\ldots, i_d, i) \in \mathrm{Cor}(\pi) \right\} \right|.
$$
\end{theorem}

As an application, the following multivariate identity was proved in \cite{ay}: 
$$
\sum_{\pi \in \mathcal{P}^{(d)}} \prod_{(i_1,\ldots, i_d, i_{d+1}) \in \mathrm{Cor}(\pi)} x_{i_1,\ldots, i_d} = \prod_{i_1,\ldots, i_d \ge 1} \left(1 - x_{i_1,\ldots, i_d} \right)^{-1}. \eqno{(\diamond)}
$$

\section{Bounds on partitions inside a simplex}\label{secsimpl}
Let $\Delta_d(k)$ be the set of positive integer points inside the $(d+1)$-dimensional simplex  $x_1 + \ldots + x_{d+1} \le k$, i.e.
$$
\Delta_d(k) := \left\{(x_1,\ldots, x_{d+1}) \in \mathbb{Z}_{+}^{d+1} :  
x_1 + \ldots + x_{d+1} \le k \right\}.
$$
Note that $\Delta_d(k)$ is a diagram of a $d$-dimensional partition of size (for $k \in \mathbb{N}$)
$$
|\Delta_d(k)| = \binom{k}{d+1} = \frac{k^{d+1}}{(d+1)!} - O(k^d) \le \frac{k^{d+1}}{(d+1)!}.
$$ 

Let 
$$
A_d(k) := \left\{\pi \in \mathcal{P}^{(d)} :  D(\pi) \subseteq \Delta_d(k)
\right\}
$$
be the set of $d$-dimensional partitions whose diagram lies inside $\Delta_d(k)$ 
and denote by $a_d(k) := |A_d(k)|$ the number of such partitions.
(For example, $a_1(k)$ is the number of partitions whose diagram lies inside the triangle $\{(x,y) \in \mathbb{Z}^{2}_+ : x + y \le k\}$, and so $a_1(k)$ is the 
Catalan number.)

\vspace{0.5em}

Let $B_d(k)$ be the set of hypermatrices $B = (b_{\mathbf{i}})_{\mathbf{i} \in \Delta_{d - 1}(k - 1)}$ (with support in $\Delta_{d - 1}(k - 1)$) with entries $b_{\mathbf{i}} \in \{0,1,2 \}$ satisfying the following property: if $b_{\mathbf{i}} = 2$ then $\mathbf{i} + \mathbf{e}_{\ell}\in \Delta_{d - 1}(k - 1)$ and $b_{\mathbf{i} + \mathbf{e}_{\ell}} = 0$ for all $\ell \in [d]$. Denote by $b_d(k) := |B_d(k)|$ the number of such hypermatrices. 

\begin{lemma}
We have: 
$$
a_d(k) \ge b_d(k).
$$
\end{lemma}

\begin{proof}
Let us show that $\Phi : B_d(k) \to A_d(k)$. 
For  $B = (b_{i_1,\ldots, i_d}) \in B_d(k)$ let $\pi = (\pi_{i_1,\ldots, i_d}) = \Phi(B)$.  
Note that $\pi \in \mathcal{P}^{(d)}$ is a $d$-dimensional partition. Let us show that $\pi \in A_d(k)$.  
We are going to prove by induction on $m = k - i_1 - \ldots - i_d \ge 0$ that $\pi_{i_1,\ldots, i_d} \le m$. For $m = 0$ the statement is obvious. For $m = 1$ we have $(i_1,\ldots, i_d) + \mathbf{e}_{\ell} \not\in \Delta_{d-1}(k-1)$ for all $\ell \in [d]$ and hence $\pi_{i_1,\ldots, i_d} = b_{i_1,\ldots, i_d} \le 1$. 
Assume now the statement holds for all $i'_1,\ldots, i'_d$ with $k - i'_1 - \cdots - i'_d < m$. To show that $\pi_{i_1,\ldots, i_d} \le m = k - i_1 - \cdots - i_d$ consider two cases. 

{\it Case 1.} If $b_{i_1,\ldots, i_d} \le 1$, then we have 
$$
\pi_{i_1,\ldots, i_d} = b_{i_1,\ldots, i_d} + \max_{\ell \in [d]} \pi_{(i_1, \ldots, i_d) + \mathbf{e}_{\ell}} \le 1 + m - 1 = m
$$
since $\pi_{(i_1,\ldots, i_d) + \mathbf{e}_{\ell}} \le m - 1$ by  assumption.

{\it Case 2.} If $b_{i_1,\ldots, i_d} = 2$, then by the condition we have $b_{j_1, \ldots, j_d}  = 0$ for all $(j_1,\ldots, j_d) = (i_1,\ldots, i_d) + \mathbf{e}_{\ell}$ and $\ell \in [d]$ and hence 
\begin{align*}
\pi_{i_1,\ldots, i_d} 
&= b_{i_1,\ldots, i_d} + \max_{(j_1,\ldots, j_d) = (i_1,\ldots, i_d) + \mathbf{e}_{\ell}} \left( b_{j_1,\ldots, j_d} + \max_{\ell' \in [d]} \pi_{(j_1,\ldots, j_d) + \mathbf{e_{\ell'}}}  \right) \le 2 + m - 2 = m
\end{align*}
since $\pi_{(j_1,\ldots, j_d) + \mathbf{e}_{\ell'}} \le m - 2$ by assumption.

So we showed that $i_1 + \ldots + i_d + \pi_{i_1,\ldots, i_d} \le k$ which means that $D(\pi) \subseteq \Delta_d(k)$ and $\pi \in {A}_d(k)$, i.e. $\Phi : B_d(k) \to A_d(k)$ as needed.
Since $\Phi$ is injective we have $|B_d(k)| \le |A_d(k)|$ which now implies the inequality. 
\end{proof}

\begin{lemma}\label{lemlb}
We have: 
$$
b_d(k) \ge \left(3 \cdot 2^{3^d - 1} \right)^{\binom{\lfloor (k-1)/3 \rfloor}{d}}. 
$$
\end{lemma}

\begin{proof}
Consider 
hypermatrices $B = (b_{\mathbf{i}})_{\mathbf{i} \in \Delta_{d-1}(k-1)} \in B_d(k)$ 
satisfying the following property: 
for $j_1,\ldots, j_d \in \mathbb{Z}_+$ in each $3 \times \cdots \times 3$ ($d$ times) box 
$$J = \{3j_1-2, 3 j_1 - 1, 3j_1 \} \times \cdots \times \{3j_d-2, 3j_d - 1, 3j_d \} \subset \Delta_{d-1}(k-1),$$ 
if $b_{\mathbf{i}} = 2$ for $\mathbf{i} \in J$ 
then $\mathbf{i} + \mathbf{e}_{\ell} \in J$  and $b_{\mathbf{i} + \mathbf{e}_{\ell}} =0$ for all $\ell \in [d]$. 
(We also set $b_{\mathbf{i}} = 0$ for $\mathbf{i}$ which do not belong to any box $J$.)
Note that for different $j_1,\ldots, j_d \in \mathbb{Z}_+$ the boxes $J$ are disjoint and the number of such boxes $J$ is equal to $\binom{\lfloor (k-1)/3 \rfloor}{d}$ (since it is the number of $j_1,\ldots, j_d$ satisfying $3j_1 + \ldots + 3j_d \le k - 1$). For every such box $J$ there are $2^{3^d}$ sub-hypermatrices $(b_{\mathbf{i}})_{\mathbf{i} \in J}$ with $b_{\mathbf{i}} \in \{0, 1\}$ and $2^{d} \times 2^{3^d - d - 1}$ sub-hypermatrices $(b_{\mathbf{i}})_{\mathbf{i} \in J}$ which contain one entry $b_{\mathbf{i}} = 2$ for $\mathbf{i} \in J$ (since  $\mathbf{i} = (i_1,\ldots, i_d) \in J$ can be chosen in $2^d$ ways where $i_{\ell} \in \{3j_{\ell} - 1, 3 j_{\ell} - 2 \}$, 
and the remaining $2^{3^d - d - 1}$ elements of $J \setminus S$ where $S = \{\mathbf{i} \} \cup \{\mathbf{i} + \mathbf{e}_{\ell}  : \ell \in [d]\}$ 
can be filled in $B$ with  $\{0,1 \}$ arbitrarily according to the described property). Therefore, we have 
$$|B_d(k)| \ge \left(2^{3^d} + 2^{d} \times 2^{3^d - d - 1} \right)^{\binom{\lfloor (k-1)/3 \rfloor}{d}}$$
as needed.
\end{proof}

\begin{remark}
Counting in $B_d(k)$ only 
hypermatrices $(b_{\mathbf{i}})_{\mathbf{i} \in \Delta_{d-1}(k-1)}$ with $b_{\mathbf{i}} \in \{0,1\}$ gives a weaker inequality $b_d(k) \ge 2^{|\Delta_{d - 1}(k - 1)|} = 2^{\binom{k - 1}{d}}.$
\end{remark}

\begin{remark}\label{aup}
To get an upper bound on $a_d(k)$, it is not difficult to show that for $k \ge d$ we have
$$
a_{d}(k) \le a_{d-1}(k-1) \cdot a_{d-1}(k-2) \cdots a_{d-1}(d).
$$
From this inequality, by induction on $d$ we obtain that 
\begin{align*}
a_d(k) \le 2^{2 \binom{k-1}{d}}. 
\end{align*}
\end{remark}

\begin{lemma}\label{lempa}
We have: 
$$\widetilde{\mathsf{p}}_{d}(n) \ge a_{d}(k) \text{ for } k  = \lfloor ((d+1)!\, n)^{1/(d+1)} \rfloor.$$ 
\end{lemma}
\begin{proof}
Note that 
$$|\Delta_{d}(k)| \le \frac{k^{d+1}}{(d+1)!} \le n$$
and so every $d$-dimensional partition in the set $A_{d}(k)$ has volume at most $n$, which proves the inequality.
\end{proof}


\begin{theorem}\label{coral}
We have:
$$
\liminf_{n \to \infty} \frac{\log \widetilde{\mathsf{p}}_{d}(n)}{n^{d/(d+1)}} \ge 
\alpha_d,  
$$
where 
$$
 \qquad \alpha_d =  \frac{(d+1) }{(d+1)!^{1/(d+1)}} \cdot \delta_d, \qquad 
 \delta_d = \liminf_{k \to \infty} \frac{\log a_d(k)}{k^d/d!} \ge \log 2 + 3^{-d} \log \frac{3}{2}. 
$$
\end{theorem}
\begin{proof}
For $k  = \lfloor ((d+1)!\, n)^{1/(d+1)} \rfloor$ we have $n^{d/(d+1)} \sim (d+1)!^{-d/(d+1)} k^d = \frac{d!}{(d+1)!^{d/(d+1)}}\cdot k^d/d! = \frac{(d+1)!^{1/(d+1)}}{(d+1)} \cdot k^d/d!$ and hence
\begin{align*}
\liminf_{n \to \infty} \frac
{\log \widetilde{\mathsf{p}}_{d}(n)}
{n^{d/(d+1)}} 
&\ge  \frac{(d+1)}{(d+1)!^{1/(d+1)}} \liminf_{k \to \infty} \frac{\log a_{d}(k)}{k^d/d!} = \frac{(d+1)}{(d+1)!^{1/(d+1)}} \cdot \delta_d = \alpha_d.  
\end{align*}
Using the above lemmas we have 
$$
\delta_d = \liminf_{k \to \infty} \frac{\log a_d(k)}{k^d/d!} \ge \liminf_{k \to \infty} \frac{\log b_d(k)}{k^d/d!} \ge \lim_{k \to \infty} \frac{\binom{\lfloor (k-1)/3 \rfloor }{d} \log \left(2^{3^d} \cdot \frac{3}{2} \right)}{k^d/d!} = \log 2 + 3^{-d} \log \frac{3}{2}.
$$
\end{proof}

\begin{remark}
One can also lower bound $\widetilde{\mathsf{p}}_{d}(n)$ via the number of partitions whose diagram lies inside the box $[m]^{d+1}$ for $m^{d+1} \le n$. Bounds on the number of partitions inside a box were obtained in \cite{ms}. However this gives a lower bound smaller than $\alpha_d$. 
\end{remark}

\begin{remark}
To compare our lower bound constant $\alpha_d$ with the constant $\gamma_d$ from MacMahon's numbers, a direct computation shows that  
$\alpha_7 > 1.47347 > \gamma_7 \approx 1.45831$. It is a routine exercise to show that $\alpha_d > \gamma_d$ for $d \ge 7$. 
\end{remark}

\begin{remark}
The constant $\alpha_d$ can be further slightly improved by giving a better lower bound in Lemma~\ref{lemlb} (e.g. by filling more $2$'s for hypermatrices in $B_d(k)$ or dividing into boxes as $J$ of other sizes), which we did not pursue here.   
\end{remark}

\begin{remark}
The upper bound on $a_d(k)$ mentioned in Remark~\ref{aup} tells that bounding the number of partitions inside the simplex $\Delta_d$ cannot give a lower bound constant on $n^{-d/(d+1)} {\log \mathsf{p}}_d(n)$ better than $\frac{(d+1)}{(d+1)!^{1/(d+1)}} \cdot 2 \log 2$. 
\end{remark}

\begin{problem}
Find 
$
\lim_{k \to \infty} ({k^d/d!})^{-1} {\log a_d(k)} \in (\log 2, \log 4]
$
and prove its existence. 
\end{problem}

\section{An upper bound via vector partitions}\label{secvec}

Recall that vector partition numbers $\mathsf{p}(n_1,\ldots, n_d)$ are defined via the multivariate generating series 
$$
\sum_{n_1,\ldots, n_d } \mathsf{p}(n_1,\ldots, n_d)\, x_1^{n_1} \cdots x_d^{n_d}  = \prod_{i_1,\ldots, i_d \ge 1} \left(1 - x_1^{i_1} \cdots x_d^{i_d} \right)^{-1}.
$$

Recall also that for a partition $\pi \in \mathcal{P}^{(d)}$, the set of {corners} of $\pi$ is defined as follows:
$$
\mathrm{Cor}(\pi) := \left\{\mathbf{i} \in \mathbb{Z}^{d+1}_+ : \mathbf{i} \in D(\pi),\, \mathbf{i} + \mathbf{e}_{\ell} \not\in D(\pi) \text{ for all } \ell \in [d] \right\}.
$$
For $\pi \in \mathcal{P}^{(d)}$ and $\ell \in [d]$, define the following statistic: 
$$
c_{\ell}(\pi) := \sum_{(i_1,\ldots, i_{d+1}) \in \mathrm{Cor}(\pi)} i_{\ell}.
$$

We show the following interpretation of vector partition numbers via $d$-dimensional partitions using the c-statistic.

\begin{lemma}\label{wp}
We have: 
$$
\mathsf{p}(n_1,\ldots, n_d) = \left| \left\{\pi \in \mathcal{P}^{(d)} : (c_1(\pi), \ldots, c_d(\pi)) = (n_1,\ldots, n_d) \right\} \right|.
$$
\end{lemma}
\begin{proof}
Firstly, note that by definition of c-statistic we have
$$
\sum_{\pi \in \mathcal{P}^{(d)}} x_1^{c_1(\pi)} \cdots x_d^{c_d(\pi)} = \sum_{\pi \in \mathcal{P}^{(d)}} \prod_{(i_1,\ldots, i_d, i_{d+1}) \in \mathrm{Cor}(\pi)} x_1^{i_1} \cdots x_d^{i_d}.
$$
Now setting $x_{i_1,\ldots, i_d} \to x_1^{i_1} \cdots x_d^{i_d}$ in the identity ($\diamond$) (in \textsection\ref{secphi}) 
we get
\begin{align*}
\sum_{\pi \in \mathcal{P}^{(d)}} \prod_{(i_1,\ldots, i_d, i_{d+1}) \in \mathrm{Cor}(\pi)} x_1^{i_1} \cdots x_d^{i_d} 
&=\prod_{i_1,\ldots, i_d \ge 1} \left(1 - x_1^{i_1} \cdots x_d^{i_d} \right)^{-1} \\
&= \sum_{n_1,\ldots, n_d} \mathsf{p}(n_1,\ldots, n_d) x_1^{n_1} \cdots x_{d}^{n_d}.
\end{align*}
Combining these identities we therefore have 
$$
\sum_{\pi \in \mathcal{P}^{(d)}} x_1^{c_1(\pi)} \cdots x_d^{c_d(\pi)} = \sum_{n_1,\ldots, n_d} \mathsf{p}(n_1,\ldots, n_d) x_1^{n_1} \cdots x_{d}^{n_d}, 
$$ 
which now implies the result.
\end{proof}

Now we show that the c-statistic does not exceed the volume which will be needed for getting the bounds below.

\begin{lemma}
Let $\pi \in \mathcal{P}^{(d)}$ and $\ell \in [d]$. Then $c_{\ell}(\pi) \le |\pi|$. 
\end{lemma}
\begin{proof}
For each corner $(i_1,\ldots, i_{d+1}) \in \mathrm{Cor}(\pi)$,  consider the set 
$$X_{\ell}(i_1,\ldots, i_{d+1}) :=  \{(i_1, \ldots, i_{\ell - 1}, i, i_{\ell + 1}, \ldots, i_{d+1}) : 1 \le i  \le i_{\ell} \} \subseteq D(\pi).$$ 
Notice that the sets $X_{\ell}(i_1,\ldots, i_{d+1})$ for $(i_1,\ldots, i_{d+1}) \in \mathrm{Cor}(\pi)$ are disjoint. Otherwise, a common element in sets $X_{\ell}$ must be produced from distinct corners with the same coordinates except for $\ell$-th one, which for $\ell \ne d+1$ is impossible by definition of corners. 
Hence, 
$$\bigsqcup_{(i_1,\ldots, i_{d+1}) \in \mathrm{Cor}(\pi)} X_{\ell}(i_1,\ldots, i_{d+1}) \subseteq D(\pi)$$
and we obtain that  
$$
c_{\ell}(\pi) = \sum_{(i_1,\ldots, i_{d+1}) \in \mathrm{Cor}(\pi)} i_{\ell} = \sum_{(i_1,\ldots, i_{d+1}) \in \mathrm{Cor}(\pi)} |X_{\ell}(i_1,\ldots, i_{d+1})| \le |D(\pi)| = |\pi|
$$
as needed. 
\end{proof}

\begin{remark}
Note that this inequality is sharp, e.g. let $\pi$ be the partition whose diagram is the simplex $\Delta_d(k)$ (defined in the previous section). Then we have $c_{\ell}(\pi) = |\pi|$ for all $\ell \in [d]$.  
\end{remark}

Let us also denote
$$
\widetilde{\mathsf{p}}(n_1, \ldots, n_d) := \sum_{i_1 \le n_1, \ldots, i_d \le n_d} \mathsf{p}(i_1,\ldots, i_d) \le n_1 \cdots n_d \cdot \mathsf{p}(n_1,\ldots, n_d). 
$$

Now we are ready to prove our main upper bound. 
\begin{theorem}\label{lempp}
We have the following inequality:
$$
\widetilde{\mathsf{p}}_d(n) \le \widetilde{\mathsf{p}}(\underbrace{n, \ldots, n}_{d \text{ times}}).
$$
\end{theorem}
\begin{proof}
By previous lemma, every $d$-dimensional partition $\pi$ with volume at most $n$ has $c_{\ell}(\pi) \le n$ for all $\ell \in [d]$.
Hence we have 
$$
\left\{\pi \in \mathcal{P}^{(d)} : |\pi| \le n \right\} \subseteq \bigsqcup_{i_1,\ldots, i_d \le n} \left\{\pi \in \mathcal{P}^{(d)} : (c_1(\pi), \ldots, c_d(\pi)) = (i_1,\ldots, i_d) \right\}
$$
Taking cardinalities and using Lemma~\ref{wp} we obtain the inequality.
\end{proof}

The following (asymptotically tight) upper bound on vector partition numbers was proved in \cite{bv} (which is derived  using its generating function). 

\begin{lemma}[\cite{bv}]\label{lembv}
We have the following bound:
$$
\log \mathsf{p}(n_1,\ldots, n_d) \le (d+1) \zeta(d+1)^{1/(d+1)} (n_1 \cdots n_d)^{1/(d+1)}.
$$
\end{lemma}

Theorem~\ref{thm:one} now follows from Theorem~\ref{coral}, Theorem~\ref{lempp} and Lemma~\ref{lembv}. Theorem~\ref{thmc} follows from 
Theorem~\ref{lempp}.

\begin{remark}
Let us see what bounds Theorem~\ref{thm:one} gives for $d  \le 3$.
For $d = 1$ the upper bound is asymptotically tight since the number of integer partitions $\mathsf{p}_1(n)$ satisfies 
$$\frac{\log \mathsf{p}_1(n)}{n^{1/2}} \sim 2 \zeta(2)^{1/2} = \beta_1 = \gamma_1 $$ according to the Hardy--Ramanujan formula. 
For $d = 2$ the number of plane partitions $\mathsf{p}_2(n)$ has the asymptotics 
$$\frac{\log \mathsf{p}_2(n)}{n^{2/3}} \sim \gamma_2 = \frac{3 \zeta(3)^{1/3}}{2^{2/3}}  \approx 2.0094$$ 
while the above bounds range about $(1.2187, 3.1898)$. 
For $d = 3$ the asymptotics of the number $\mathsf{p}_3(n)$ of solid partitions is unknown, and the bounds we obtain give that
$$
1.2797 
< \frac{\log \mathsf{p}_{3}(n)}{n^{3/4}} < 
4.0799
$$
for sufficiently large $n$. It was estimated in \cite{dg} that 
$n^{-3/4}{\log \mathsf{p}_{3}(n)}{} \sim 1.822 \pm 0.001$.
\end{remark}

\subsection{An upper bound via MacMahon's numbers}\label{secmac}
For $\pi \in \mathcal{P}^{(d)}$, the {\it corner-hook volume} statistic $|\cdot|_{ch} : \mathcal{P}^{(d)} \to \mathbb{N}$ is defined as follows:
$$
|\pi|_{ch} := \sum_{(i_1,\ldots , i_{d+1})\, \in\, \mathrm{Cor}(\pi)} (i_1 +\ldots + i_d - d + 1) = c_{1}(\pi) + \cdots + c_{d}(\pi) - (d-1)\cdot |\mathrm{Cor}(\pi)|.
$$
MacMahon's numbers have the following interpretation via $d$-dimensional partitions obtained in \cite{ay}:
$$
\mathsf{m}_d(n) = \left| \left\{\pi \in \mathcal{P}^{(d)} : |\pi|_{ch} = n \right\} \right|.
$$
Similarly denote 
$
\widetilde{\mathsf{m}}_d(n) := \sum_{i \le n}^{} \mathsf{m}_d(i). 
$
Then the same proof as above shows that $|\pi|_{ch} \le d\, |\pi| - (d - 1) |\mathrm{Cor}(\pi)| \le |\pi| - (d - 1)$ and 
$$
\widetilde{\mathsf{p}}_d(n) \le \widetilde{\mathsf{m}}_d(dn - d + 1).
$$
In particular, we also have the inequality $\mathsf{p}_d(n) \le (dn - d + 1)\, \mathsf{m}_d(dn - d + 1)$.


\


\end{document}